\documentclass[a4paper]{rmmcart}

\usepackage{amsmath,amssymb,amsthm}
\usepackage[dvipsnames]{xcolor}

\newtheorem{theorem}{Theorem}[section]
\newtheorem{proposition}[theorem]{Proposition}
\newtheorem{corollary}[theorem]{Corollary}

\newtheorem{algorithm}[theorem]{Algorithm}

\theoremstyle{definition}

\newtheorem{remark}[theorem]{Remark}
\newtheorem{definition}[theorem]{Definition}
\newtheorem{example}[theorem]{Example}

\def\OO{{\mathcal{O}}}
\def\QQ{\mathbb{Q}}
\def\ZZ{\mathbb{Z}}
\def\X{{\mathbb{X}}}
\def\F{{\mathcal{F}}}
\def\AA{\mathbb{A}}
\def\M{\mathfrak{M}}
\def\Q{{\mathfrak{Q}}}
\def\m{\mathfrak{m}}

\let\epsilon=\varepsilon

\def\phi{{\varphi}}
\let\Psi=\varPsi
\let\Phi=\varPhi
\let\theta=\vartheta

\DeclareSymbolFont{CMMfont}{OML}{cmm}{m}{it}
\DeclareMathSymbol{\Varrho}{3}{CMMfont}{37}
\def\rho{{\mathop{\Varrho}\,}}

\def\NF{\mathop{\rm NF}\nolimits}

\def\LF{\mathop{\rm LF}\nolimits}
\def\DF{\mathop{\rm DF}\nolimits}
\def\HF{\mathop{\rm HF}\nolimits}
\def\HFa{\mathop{\rm HF}\nolimits^a}
\def\ri{\mathop{\rm ri}\nolimits}

\def\Mat{\mathop{\rm Mat}\nolimits}

\def\Spec{\mathop{\rm Spec}\nolimits}

\def\charac{\mathop{\rm char}\nolimits}

\def\Rad{\mathop{\rm Rad}\nolimits}
\def\Ker{\mathop{\rm Ker}\nolimits}

\def\gr{\mathop{\rm gr}\nolimits}
\def\grF{{\mathop{\rm gr}\nolimits_{\mathcal{F}}}}
\def\ord{\mathop{\rm ord}\nolimits}

\def\Jac{\mathop{\rm Jac}\nolimits}

\def\Fitt{\mathop{\rm Fitt}\nolimits}

\let\To=\longrightarrow

\def\hom{^{\rm hom}}
\def\tr{^{\,\rm tr}}
\def\tfrac #1#2{{\textstyle\frac{#1}{#2}}}

\def\cocoa{\mbox{\rm
  C\kern-.13em o\kern-.07 em C\kern-.13em o\kern-.15em A}}
\def\apcocoa{\mbox{\rm
A\kern-0.13em p\kern -0.07em C\kern-.13em o\kern-.07 em C\kern-.13em
o\kern-.15em A}}

\title[Algorithms for Checking Complete Intersections]{Algorithms 
for Checking Zero-Dimensional Complete Intersections}

\author[M. Kreuzer]{Martin Kreuzer}
\address{Fakult\"at f\"ur Informatik und Mathematik, Universit\"at
Passau, D-94030 Passau, Germany}
\email{Martin.Kreuzer@uni-passau.de}

\author[L.N. Long]{Le Ngoc Long}
\address{Fakult\"at f\"ur Informatik und Mathematik, Universit\"at Passau,
D-94030 Passau, Germany and Department of Mathematics, 
Hue University of Education, 34 Le Loi, Hue, Vietnam}
\email{nglong16633@gmail.com}

\author[L. Robbiano]{Lorenzo Robbiano}
\address{Dipartimento di Matematica, Universit\`a di Genova,
Via Dodecaneso 35, I-16146 Genova, Italy}
\email{lorobbiano@gmail.com}

\date{\today}

\keywords{zero-dimensional affine algebra, zero-dimensional scheme,
complete intersection, locally complete intersection, 
strict complete intersection, Fitting ideal, border basis}

\subjclass{Primary 13C40, Secondary  14M10, 13H10, 13P99, 14Q99}

\begin{document}

\begin{abstract}
Given a 0-dimensional affine $K$-algebra $R=K[x_1,\dots,x_n]/I$,
where $I$ is an ideal in a polynomial ring $K[x_1,\dots,x_n]$
over a field~$K$, or, equivalently, given a 
0-dimensional affine scheme, we construct effective algorithms
for checking whether $R$ is a complete intersection at a maximal
ideal, whether $R$ is  locally a complete intersection, and whether
$R$ is a strict complete intersection. These algorithms are based
on Wiebe's characterization of \hbox{0-di}\-men\-sional local complete
intersections via the 0-th Fitting ideal of the maximal ideal.
They allow us to detect which generators of~$I$ form a regular sequence
resp.\ a strict regular sequence, and they work over
an arbitrary base field~$K$. Using degree filtered border bases,
we can detect strict complete intersections in certain families
of 0-dimensional ideals.
\end{abstract}

\maketitle

%
%

\section{Introduction}

Regular sequences and complete intersections play a fundamental role in 
Commutative Algebra and Algebraic Geometry. 
Given an ideal~$I$ in a polynomial ring $P=K[x_1,\dots,x_n]$
over a field~$K$, e.g., the vanishing ideal of an affine or 
projective scheme, it is therefore an important algorithmic task
to check whether~$I$ is a complete intersection ideal, that is,
whether~$I$ can be generated by ${\rm ht}(I)$ polynomials.
Equivalently, we call $R=P/I$ a complete intersection ring in this case.
Several approaches have been developed to tackle this problem
effectively for special classes of ideals.

For instance, if the ring $R=P/I$ is local, based on a description
of the structure of the $R$-algebra ${\rm Tor}^R(K,K)$ 
in~\cite{T57}, characterizations of the complete intersection
property using the Hilbert series of ${\rm Tor}^R(K,K)$
were developed in~\cite{As59} and~\cite{GL}. However, this approach
requires the calculation of the Hilbert series of a non-commutative
algebra. Furthermore, in the local setting, 
complete intersections were characterized in~\cite{V67} using 
the freeness of the conormal module $I/I^2$
and the finiteness of the projective dimension of~$I$. Again, these
conditions are not easy to check algorithmically.
In a similar vein, if the base field~$K$ has characteristic
zero, one can use techniques based on the K\"ahler differential
module $\Omega^1_{R/K}$, as in~\cite{AH94}, or on K\"ahler
differents, as in~\cite{KL},
but they are neither general enough for our purposes nor do they
lend themselves to a generalization for families of ideals.
Finally, let us mention that effective algorithms have been
developed in~\cite{BGS} for checking the complete intersection property of
toric ideals defining affine semigroup rings.

In this paper we are interested in algorithms for 
checking the locally complete intersection property and the strict
complete intersection property for a \hbox{0-dimen}\-sion\-al ideal $I$
in $P=K[x_1,\dots,x_n]$, where $K$ is an arbitrary field.
In other words, we want to check these properties for a
0-dimensional affine \hbox{$K$-al}\-ge\-bra of the form $P/I$.
In the language of Algebraic Geometry, the scheme
$\X=\Spec(P/I)$ is then a 0-dimensional affine scheme,
and $R = P/I$ is its affine coordinate ring.
Notice that every 0-dimensional projective scheme can be
embedded into a basic affine open subset of~$\mathbb{P}^n$
after possibly extending the field~$K$ slightly.
Then we can readily move between the languages of affine and
projective geometry by using homogenization and dehomogenization.
For reasons which will become clear shortly, we prefer the affine
setting in this paper. In any case, 0-dimensional affine $K$-algebras
$R=P/I$ have been used in many areas besides Algebraic Geometry,
for instance in algebraic statistics, algebraic biology, etc.

Our main algorithms check whether a 0-dimensional affine
\hbox{$K$-al}\-ge\-bra $R=P/I$ is locally a complete intersection or
a strict complete intersection. They are based on a characterization
of local complete intersections by H.~Wiebe in~\cite{Wie}, where it is
shown that a 0-dimensional local ring is a complete intersection
if and only if the 0-th Fitting ideal of its maximal ideal is non-zero.
Given a \hbox{0-dimen}\-sional affine $K$-algebra $R=P/I$
whose defining ideal~$I$ has a primary decomposition
$I=\Q_1 \cap \cdots \cap \Q_s$, this characterization can be applied
right off the bat, because the localizations are of the form 
$P/\Q_i$, and hence again 0-dimensional local affine $K$-algebras.
For checking the strict complete intersection property, we use
the graded ring $\grF(R)$ with respect to the degree filtration~$\F$
instead. It has a presentation $\grF(R) = P/\DF(I)$, where
the degree form ideal $\DF(I)$ is homogeneous and 0-dimensional
(see~\cite{KR2}, Def.~4.2.13).

Now let us describe the contents of this paper in detail.
In Section~\ref{Zero-Dimensional Affine Schemes} 
we start by recalling some basic properties
of a $0$-dimensional scheme~$\X$ embedded in an affine 
space $\mathbb{A}^n_K$, where~$K$ is an arbitrary field.
In particular, besides recalling the degree filtration
and the affine Hilbert function of the affine coordinate
ring $R_\X = P/I_\X$ of~$\X$, where $I_\X$ is the
vanishing ideal of~$\X$ in $P=K[x_1,\dots,x_n]$ and~$K$
is a field, we recall the degree form ideal $\DF(I_\X)$,
the associated graded ring $\grF(R_\X)$, and mention their
connection to the Rees algebra $\mathcal{R}_\F(R_\X)$
and to Macaulay bases of~$I_\X$.

Then the main part of the paper starts in Section~\ref{Local 
Complete Intersections}. After defining what we mean by
a complete intersection at a maximal ideal and by the property of being
locallly a complete
intersection, we recall Wiebe's characterization mentioned
above (see~\cite{Wie}, Satz~3). The central tool for computing
the 0-th Fitting ideal needed in this characterization is
given in Proposition~\ref{prop:gensyz}. For a maximal ideal~$\M$
of~$P$ and an $\M$-primary ideal~$\Q$, it suffices to write 
the generators of~$\Q$ in terms of a regular sequence
generating~$\M$ and to compute the maximal minors of the
coefficient matrix. As an immediate consequence, we obtain 
Algorithm~\ref{alg:CheckLCI} for checking if a 0-dimensional
affine $K$-algebra $R_\X=P/I_\X$ is a complete intersection at a maximal ideal,
and by combining this with the computation of a primary
decomposition of~$I_\X$, we get an algorithm for checking
whether $R_\X$ is locally a complete intersection.
Moreover, the non-zero minors provide us with regular sequences
generating $\Q\, P_\M$, as Proposition~\ref{prop:whichgens} shows.
Suitable examples at the end of the section
illustrate the merits and some hidden features of these algorithms.

In Section~\ref{Strict Complete Intersections via Macaulay Bases}  
further notions are recalled. More precisely, strict regular sequences, 
strict Gorenstein rings, and strict complete intersections enter the game.
In particular, the ideal $I_\X$ is called a strict complete intersection
if its degree form ideal $\DF(I_\X)$ is generated by a homogeneous
regular sequence. Since the graded ring $\grF(R_\X) = P/\DF(I_\X)$
is 0-dimensional and local, we can use Proposition~\ref{prop:gensyz}
and Wiebe's result again to construct Algorithm~\ref{alg:CheckSCI}
for checking the strict complete intersection property.
Moreover, we also get a description of which generators of~$I_\X$
form a strict regular sequence (see Corollary~\ref{cor:whichhomoggens})
and illustrate everything via some explicit examples.

In the last section of the paper, we present a second algorithm
for checking the strict complete intersection property via border bases
which generalizes to certain families of 0-dimensional ideals.
Based on the notion of a degree filtered $K$-basis of~$R_\X$,
we define degree filtered $\OO$-border bases. They have several
nice characterizations, the most useful one here being the
property that the degree forms of the border basis polynomials
form an $\OO$-border basis of the degree form ideal
(see Proposition~\ref{prop:CharDFBB}). Then we obtain
Algorithm~\ref{alg:CheckSCI2} which checks for the strict
complete intersection property using a degree filtered
border basis. This version allows us to detect all  strict complete
intersections within certain families of 0-dimensional ideals, as 
illustrated by Example~\ref{ex:SCI22} and applied further
in~\cite{KLR3}. Finally, in Remark~\ref{rem:Jac}, we compare 
the algorithms of this paper with the methods based on Jacobian matrices,
K\"ahler differentials, and K\"ahler differents mentioned above.

All examples in this paper were computed using the computer algebra
system \cocoa\ (see~\cite{CoCoA}).
Unless explicitly stated otherwise, we adhere to the notation
and definitions provided in~\cite{KR1}, \cite{KR2}, and~\cite{KR3}.

\bigbreak
%
%

\section{Zero-Dimensional Affine Schemes}
\label{Zero-Dimensional Affine Schemes}

In the following we always work over an arbitrary field~$K$
and let $\AA^n_K$ be the affine $n$-space over~$K$.
We fix a coordinate system, so that
the affine coordinate ring of~$\AA^n_K$ is given by the
polynomial ring $P=K[x_1,\dots,x_n]$. Thus a
0-dimensional subscheme~$\X$ of~$\AA^n_K$ is defined by a
0-dimensional ideal~$I_\X$ in $P\!$, and its affine coordinate
ring is $R_\X=P/I_\X$. Consequently, the vector space dimension
$\mu=\dim_K(R_\X)$ is finite and equal to the length of the
scheme~$\X$.

Since we are keeping the coordinate system fixed
at all times, we have further invariants of~$\X$.
Recall that the {\bf degree filtration} $\widetilde{\mathcal{F}}
= (F_iP)_{i\in\ZZ}$ on~$P$ is given by $F_iP =
\{f\in P\setminus \{0\} \mid \deg(f)\le i\} \cup \{0\}$
for all $i\in\ZZ$. Then the induced filtration
$\F = (F_iR_\X)_{i\in\ZZ}$, where $F_i R_\X =
F_iP /(F_iP\cap I_\X)$, is called the {\bf degree filtration}
on~$R_\X$. It is easy to see that the degree filtration
on~$R_\X$ is increasing, exhaustive and {\bf orderly} in the
sense that every element $f\in R_\X \setminus \{0\}$ has an
{\bf order} $\ord_\F(f)=\min\{i\in\ZZ \mid f\in F_i R_\X
\setminus F_{i-1}R_\X\}$.

\begin{definition}\label{def:Hilb}
Let~$\X$ be a 0-dimensional subscheme of~$\AA^n_K$
as above.
\begin{enumerate}
\item[(a)] The map $\HFa_\X: \ZZ \To \ZZ$ given by $i\mapsto
\dim_K(F_i R_\X)$ is called the {\bf affine Hilbert function} of~$\X$.

\item[(b)] The number $\ri(R_\X)= \min \{i\in\ZZ \mid
\HFa_\X(j)=\mu$ for all $j\ge i\}$ is called the {\bf regularity
index} of~$\X$.

\item[(c)] The first difference function $\Delta\HFa_\X (i) =
\HFa_\X(i) - \HFa_\X(i-1)$ of~$\HFa_\X$ is called the
{\bf Castelnuovo function} of~$\X$, and the number $\Delta_\X =
\Delta \HFa_\X(\ri(R_\X))$ is the {\bf last difference} of~$\X$.
\end{enumerate}
\end{definition}

It is well-known that $\HFa_\X$ satisfies $\HFa_\X(i)=0$ for
$i<0$ and
$$
1 = \HFa_\X(0) < \HFa_\X(1) < \cdots < \HFa_\X(\ri(R_\X)) = \mu
$$
as well as $\HFa_\X(i)=\mu$ for $i\ge \ri(R_\X)$. The affine Hilbert
function of~$\X$ is related to the following objects.

\begin{definition}\label{def:degreeform}
Let $\X$ be a 0-dimensional subscheme of~$\AA^n_K$ as above.
\begin{enumerate}
\item[(a)] For every polynomial $f\in P\setminus \{0\}$, its
homogeneous component of highest degree is called the
{\bf degree form} of~$f$ and is denoted by $\DF(f)$.

\item[(b)] The ideal
$\DF(I_\X)= \langle \DF(f) \mid f\in I_\X \setminus \{0\}\rangle$
is called the {\bf degree form ideal} of~$I_\X$.

\item[(c)] The ring $\gr_{\widetilde{\F}}(P) =
\bigoplus_{i\in\ZZ} F_i P / F_{i-1}P$
is called the {\bf associated graded ring} of~$P$ with respect
to~$\widetilde{\F}$.

\item[(d)] The ring $\gr_\F(R_\X)= 
\bigoplus_{i\in\ZZ} F_i R_\X / F_{i-1} R_\X$ is called the 
{\bf associated graded ring} of~$R_\X$ with respect to~$\F$.

\item[(e)] For an element $f\in R_\X \setminus \{0\}$ of order
$d=\ord_\F(f)$, the residue class
$\LF(f) = f+ F_{d-1} R_\X$ in $\grF(R_\X)$ is called the
{\bf leading form} of~$f$ with respect to~$\F$.
\end{enumerate}
\end{definition}

We observe that in our setting the associated graded ring
$\grF(R_\X)$ is a \hbox{0-dimensional} local ring whose 
maximal ideal is generated by the residue classes of the indeterminates.
Its $K$-vector space dimension is given by
$$
\dim_K(\grF(R_\X) \;=\; {\textstyle\sum\limits_{i=0}^\infty} 
\dim_K(F_i R_\X / F_{i-1}R_\X) \;=\; \dim_K(R_\X)
$$
For actual computations involving the associated graded ring
and leading forms, we can use the following observations.

\begin{remark}\label{rem:ReprGr}
Notice that there is a canonical isomorphism of graded
\hbox{$K$-al}\-ge\-bras $\gr_{\widetilde{\F}}(P)\cong P$
which allows us to identify $\grF(R_\X) \cong P/\DF(I_\X)$
(see~\cite{KR2}, Props.~6.5.8 and 6.5.9).

In order to represent the leading form of a non-zero element
$f\in R_\X$ as a residue class in $P/\DF(I_\X)$, we first have
to represent~$f$ by a polynomial $F\in P$ with $\deg(F)=
\ord_\F(f)$. This can, for instance, be achieved by
taking any representative~$\widetilde{F}\in P$ of~$f$
and computing the normal form $F=\NF_{\sigma,I_\X}(\widetilde{F})$
with respect to a degree compatible term ordering~$\sigma$.
Then the degree form $\DF(F)$ represents the leading form
$\LF_\F(f)$ with respect to the isomorphism $\grF(R_\X) \cong
P/\DF(I_\X)$.
\end{remark}

The affine Hilbert function of~$\X$ and the (usual) Hilbert function
of~$\grF(R_\X)$ are connected by
$$
\HFa_\X(i) \;=\; {\textstyle\sum\limits_{j=0}^i} \HF_{\grF(R_\X)}(j)
\hbox{\quad \rm and\quad }
\Delta \HFa_\X(i) \;=\; \HF_{\grF(R_\X)}(i)
$$
for all $i\ge 0$. Since the Hilbert function of~$\grF(R_\X)$ can be
calculated using a suitable Gr\"obner basis of~$I_\X$ (see~\cite{KR2},
Section 4.3), this formula allows us to compute the affine Hilbert
function of~$\X$.

The ring~$R_\X$ and its graded ring $\grF(R_\X)$ are connected
by the following flat family. (For further details and proofs,
see~\cite{KR2}, Section 4.3.B.)

\begin{remark}\label{ReesAlg}
Let $x_0$ be a further indeterminate.
\begin{enumerate}
\item[(a)] The ring ${\mathcal{R}_{\widetilde{\F}}(P) = \bigoplus_{i\in\ZZ}
F_i P\cdot x_0^i}_{\mathstrut}$ is called the {\bf Rees algebra}
of~$P$ with respect to~${\widetilde{\F}}$. Every non-zero homogeneous
element of the Rees algebra is of the form $f\cdot x_0^{d+j}$
with $d=\deg(f)$ and $j\ge 0$. By identifying it with the polynomial
$f\hom\cdot x_0^j$, we obtain an isomorphism of graded $K[x_0]$-algebras
$\mathcal{R}_{\widetilde{\F}}(P) \cong K[x_0,x_1,\dots,x_n] =: \overline{P}$.
Here $f\hom $ is the usual homogenization $f\hom = x_0^{\deg(f)}
\cdot f(\tfrac{x_1}{x_0},\dots,\tfrac{x_n}{x_0})$.

\item[(b)] Similarly, using the degree filtration~$\F$ on~$R_\X$, we have the
{\bf Rees algebra} $\mathcal{R}_\F(R_\X) = \bigoplus_{i\in\ZZ}
F_i R_\X \cdot x_0^i$. Now we let $I_\X\hom = \langle f\hom \mid f\in I_\X
\setminus \{0\}\rangle$ be the {\bf homogenization} of~$I_\X$. Then the above
isomorphism induces an isomorphism of graded $K[x_0]$-algebras
$\mathcal{R}_\F(R_\X) \cong \overline{P} / I_\X\hom =: R_\X\hom$.

Geometrically speaking, the ring $R_\X\hom$ is the homogeneous coordinate
ring of the 0-dimensional scheme obtained by embedding
$\X=\Spec(P/I_\X) \subset \AA^n_K$
into the projective $n$-space via $\AA^n_K \cong D_+(x_0)
\subset \mathbb{P}^n$.

\item[(c)] The ring $R_\X\hom$ is a free (and hence flat) $K[x_0]$-module.
The elements $x_0$ and $x_0-1$ are non-zerodivisors for $R_\X\hom$,
and we have canonical isomorphisms
$$
R_\X\hom / \langle x_0-1\rangle \cong R_\X \hbox{\qquad \rm and\qquad}
R_\X\hom / \langle x_0\rangle \cong \grF(R_\X)
$$
In other words, the algebra homomorphism $K[x_0] \To R_\X\hom$
defines a flat family whose special fiber is $\grF(R_\X)$
and whose general fiber is~$R_\X$.

\item[(d)] A set $\{f_1,\dots,f_s\}\subseteq I_\X$ with
$\DF(I_\X)=\langle \DF(f_1),\dots,\DF(f_s)\rangle$ is called a
{\bf Macaulay basis} of~$I_\X$. In this case the polynomials
generate~$I_\X$ and we have $I_\X\hom = \langle f_1\hom,\dots,
f_s\hom\rangle$ (see~\cite{KR2}, Thm.~4.3.19).
\end{enumerate}
\end{remark}

\bigbreak
%
%

\section{Locally Complete Intersections}
\label{Local Complete Intersections}

In  this section we analyse a property
of a 0-dimensional subscheme~$\X$ of~$\AA^n_K$, namely the
property of being locally a complete intersection. 
For this purpose, we continue to
use the definitions and notation introduced in the preceding section.
Recall that a 0-dimensional local ring of the form $P_\M/I$ with
a field~$K$, a maximal ideal~$\M$, and a 0-dimensional ideal~$I$,
is called a {\bf complete intersection} if~$I$ can be generated
by a regular sequence of length~$n$ in $P_\M$.
In our setting, the following versions of this notion
will be considered.

\begin{definition}\label{def:locCI}
Let $\X$ be a 0-dimensional subscheme of~$\AA^n_K$ as above, let
$I_\X=\Q_1\cap \cdots\cap \Q_s$ be the primary decomposition of~$I_\X$,
and let $\m_1,\dots,\m_s$ be the maximal ideals of~$R_\X$.
\begin{enumerate}
\item[(a)] The scheme~$\X$ is called a {\bf complete intersection
at~$\m_i$} if the local ring  $(R_\X)_{\m_i} \cong P/\Q_i$
is a complete intersection.

\item[(b)] We say that~$\X$ is a {\bf locally complete
intersection scheme} if it is a complete intersection
at~$\m_i$ for each $i\in\{1,\dots,s\}$.
\end{enumerate}
\end{definition}

In the following we are looking for characterizations of 0-dimensional
locally complete intersection schemes which can be checked
algorithmically.
Our approach is inspired by the following result of H.~Wiebe
(cf.~\cite{Wie}, Satz~3). Here the $i$-th {\bf Fitting ideal} of a
module~$M$ is denoted by $\Fitt_i(M)$. (For a definition and 
basic properties of Fitting ideals, see~\cite{Ku2}, Appendix~D.)

\begin{proposition}\label{prop:Wiebe}
Let $R$ be a local ring with maximal ideal~$\m$.
Then the following conditions are equivalent.
\begin{enumerate}
\item[(a)] The ring $R$ is a 0-dimensional complete intersection.

\item[(b)] We have $\Fitt_0(\m)\ne \langle 0\rangle$.
\end{enumerate}
\end{proposition}

Now we turn this characterization into an algorithm 
by specifying a way to compute the 0-th Fitting ideal of the
maximal ideal in the setting of this paper.
The next proposition provides the central tool for this algorithm.
Recall that a maximal ideal in the polynomial ring can always
be generated by a regular sequence of length~$n$ (see~\cite{KR3},
Cor.~5.3.14).

\begin{proposition}\label{prop:gensyz}
Let $P=K[x_1,\dots,x_n]$, let~$\M$ be a maximal ideal of~$P$,
and let $\{g_1,\dots,g_n\}$ be a regular sequence which generates~$\M$.
Then let $\Q\subset P$ be an $\M$-primary ideal, let $\{f_1,\dots,f_r\}$
be a system of generators of~$\Q$, let $R=P/\Q$, and let $\m=\M/\Q$.
For $j=1,\dots,r$, write $f_j=\sum_{i=1}^n a_{ij}g_i$, and form the
matrix $W\in\Mat_{n,r}(P)$ of size $n\times r$
whose columns are given by $\sum_{i=1}^n a_{ij}e_i$ for $j=1,\dots,r$.
Then the 0-th Fitting ideal $\Fitt_0(\m)$ is generated by
the residue classes in~$R$ of the minors of order~$n$ of the 
matrix~$W$.
\end{proposition}

\begin{proof}
Notice that we have $r\ge n$, since the height of~$\Q$ is~$n$.
A presentation of the $P$-module $\m$ is given by
$\epsilon:\; P^n \longrightarrow \m$ with $\epsilon(e_i)=\bar{g}_i$
for $i=1,\dots,n$.
Let us first show that the kernel of~$\epsilon$ is generated
by the elements in the set
$$
\{g_ie_j -g_j e_i \mid 1\le i<j\le n\} \;\cup\;
\{ {\textstyle\sum\nolimits_{i=1}^n} a_{ij}e_i \mid j=1,\dots,r\} \eqno{(\ast)}
$$
It is clear that all these elements are contained in $\Ker(\epsilon)$.
Conversely, let $\sum_{i=1}^n b_ie_i\in \Ker(\epsilon)$.
Then $\sum_{i=1}^n b_i g_i \in \Q$ implies that there exist
polynomials $c_k\in P$ such that $\sum_{i=1}^n b_i g_i =
\sum_{k=1}^r c_k f_k$. Then we have the equality 
$\sum_{i=1}^n (b_i -\sum_{k=1}^r c_k a_{ik})\, g_i=0$, and the fact
that the syzygy module of a regular sequence is generated by
the trivial syzygies implies that the element $\sum_{i=1}^n (b_i -
\sum_{k=1}^r c_k a_{ik})\, e_i$ is contained in the module generated
by the first set in $(\ast)$. Therefore
$\sum_{i=1}^n b_ie_i$ is in the module generated by the
two sets in~$(\ast)$.

It follows that the ideal $\Fitt_0(\m)$ is generated by
the residue classes in $R=P/\Q$ of the minors of order~$n$ 
of the matrix $V$ whose
columns are given by the  tuples in~$(\ast)$.
It remains to show that all minors of~$V$ which involve a column
corresponding to a trivial syzygy belong to $\Q$.

For this, we may assume that there exist a number
${k\in\{1,\dots,n-1\}}$ and indices $1\le j_1< \cdots < j_k\le n$
such that the columns of type $(a_{1j},\dots,a_{nj})\tr$
involved in the minor are exactly the columns $j_1,\dots,j_k$,
and the remaining columns correspond to trivial syzygies.
W.l.o.g.\ let the last column correspond to the trivial
syzygy $g_\lambda e_\kappa - g_\kappa e_\lambda$ 
with $1\le \kappa < \lambda \le n$. Then the
tuple $(g_1,\dots,g_n)$ is a solution of the linear system
$$
\begin{pmatrix}
a_{1 j_1} & \cdots & \cdots & \cdots & \cdots & \cdots & a_{n j_1}\\
\vdots &&& \vdots &&& \vdots \\
a_{1 j_k} & \cdots & \cdots & \cdots & \cdots & \cdots & a_{n j_k}\\
0  & g_{\lambda'} & \cdots & -g_{\kappa'} & \cdots  & \cdots & 0\\
\vdots &&& \vdots &&& \vdots \\
0 & \cdots & g_{\lambda} & \cdots & -g_{\kappa} & \cdots & 0
\end{pmatrix}
\cdot \begin{pmatrix}
z_1\\ \vdots \\ z_n
\end{pmatrix} \;=\; \begin{pmatrix}
f_{j_1}\\ \vdots \\ f_{j_k} \\ 0\\ \vdots \\ 0
\end{pmatrix}
$$
where the rows of the left-hand side matrix are the columns of~$V$
corresponding to the chosen minor.
If the determinant of the matrix $\mathcal{M}$ on the left-hand side is zero,
there is nothing to show. Hence we may assume that it is non-zero,
and since~$P$ is an integral domain, the system has a unique solution
$(g_1,\dots,g_n)$ given by Cramer's Rule.
In particular, we have
$$
\det(\mathcal{M}) \cdot g_\lambda \;=\;
\det\,\begin{pmatrix}
a_{1 j_1} & \cdots & \cdots & \cdots & f_{j_1} & \cdots & a_{n j_1}\\
\vdots && \vdots && \vdots && \vdots\\
a_{1 j_k} & \cdots & \cdots & \cdots & f_{j_k} & \cdots & a_{n j_k}\\
0  & g_{\lambda'} & \cdots & -g_{\kappa'} & 0 & \cdots & 0\\
\vdots && \vdots && \vdots && \vdots \\
0 & \cdots & g_{\lambda} & \cdots & 0 & \cdots & 0
\end{pmatrix}
$$
Developing the determinant on the right-hand side first
via the last row and then via the $\lambda$-th column yields
$\det(\mathcal{M})\cdot g_\lambda \in g_\lambda \cdot 
\langle f_{j_1},\dots,f_{j_k}\rangle$. 
Since $g_\lambda$ is a non-zerodivisor in~$P$,
we can cancel it. Then we obtain $\det(\mathcal{M})\in 
\langle f_{j_1},\dots,f_{j_k} \rangle \subseteq \Q$, and 
this finishes the proof.
\end{proof}

By applying this proposition to a local ring $P/\Q$, we
can formulate an algorithm to check whether~$P/\Q$ is a 
complete intersection.

\begin{algorithm}{\bf (Checking a Complete Intersection 
at~$\mathfrak{m}$)}\label{alg:CheckLCI}\\
Let $P=K[x_1,\dots,x_n]$, let~$\M$ be a maximal ideal of~$P$,
and let $\{g_1,\dots,g_n\}$ be a regular sequence which generates~$\M$.
Then let $\Q\subset P$ be an $\M$-primary ideal, 
and let $\{f_1,\dots,f_r\}$ be a system of generators of~$\Q$.
The following instructions define an algorithm which
checks whether~$\Q P_\M$ is generated by a regular sequence
and returns the corresponding Boolean value.

\begin{enumerate}

\item[(1)] For $j=1,\dots,r$, write $f_j=
\sum_{k=1}^n a_{kj} g_k$.

\item[(2)] Form the matrix~$W$ of size $n \times r$
whose columns are given by $\sum_{k=1}^n a_{kj} e_k$.

\item[(3)] Calculate the tuple of residue classes in $P/\Q$ of
the minors of order~$n$ of~$W.$  
If the result is different from $(0,\dots,0)$,
return {\tt TRUE}. Otherwise, return {\tt FALSE} and stop.
\end{enumerate}
\end{algorithm}

\begin{proof}
In view of Proposition~\ref{prop:Wiebe}, it suffices to show that
Steps~(2) and~(3) calculate the ideal $\Fitt_0(\M/\Q)$.
This follows from Proposition~\ref{prop:gensyz}.
\end{proof}

\begin{remark}
For checking if an ideal in $P$ is maximal, and if this is the case, 
to compute a regular sequence which generates it, we refer the reader to the
algorithms in~\cite{KR3}, Ch.~5 and~\cite{ABPR}.
\end{remark}

A solution for the task of checking whether a $0$-dimensional scheme is  
locally a complete intersection scheme follows easily.

\begin{algorithm}{\bf (Checking Locally Complete 
Intersections)}\label{alg:FullCheckLCI}\\
Let $\X$ be a 0-dimensional scheme in~$\AA^n_K$, and let
$P/ I_\X$ be the affine coordinate ring of~$\X$. 
The following instructions define an algorithm which
checks whether~$\X$ is locally a complete intersection
and returns the corresponding Boolean value.
\item[(1)] Compute the primary decomposition
$I_\X = \Q_1 \cap \cdots \cap \Q_s$ of the ideal~$I_\X$.

\item[(2)] For $i=1,\dots,s$, check whether $P/\Q_i$
is a complete intersection ring using Algorithm~\ref{alg:CheckLCI}.
If the answer is {\tt TRUE} for every $i\in\{1,\dots s\}$, return {\tt TRUE}.
Otherwise, return {\tt FALSE}.
\end{algorithm}

If the result of Algorithm~\ref{alg:CheckLCI} is {\tt TRUE}, then
a more careful examination of Step~(4) produces the following
extra information.

\begin{proposition}\label{prop:whichgens}
In the setting of Algorithm~\ref{alg:CheckLCI}, the following conditions 
are equivalent.
\begin{enumerate}
\item[(a)] The residue class in $P/\Q$ of the
minor involving columns $(j_1,\dots, j_n)$ of the matrix~$W$
is non-zero. 

\item[(b)] The polynomials $f_{j_1},\dots,f_{j_n}$
form a  regular sequence in $P_\M$ which generates~$\Q P_\M$.
\end{enumerate}
Moreover, if $f_{j_1},\dots,f_{j_n}$ satisfy the equivalent conditions 
above, and if we have $\Rad(\langle f_{j_1},\dots,f_{j_n}\rangle)=\M$, 
then $\langle f_{j_1},\dots,f_{j_n}\rangle = \Q$.
\end{proposition}

\begin{proof}
Since (b)$\Rightarrow$(a) follows immediately from 
Algorithm~\ref{alg:CheckLCI}, let us prove that (a)$\Rightarrow$(b).
First, we show that $\Q P_\M$ is generated by
$\{f_{j_1},\dots,f_{j_n}\}$. By Algorithm~\ref{alg:CheckLCI}, we know that
there exist polynomials $h_1,\dots,h_n \in P$ such that 
$(h_1,\dots,h_n)$ is a regular sequence in $P_\M$ which generates
$\Q P_\M$.
For $k=1,\dots,n$, we write 
$h_k = \sum_{\ell=1}^n b_{\ell k}g_\ell$
and $f_{j_k}= \sum_{\ell=1}^n \tfrac{c_{\ell k}}{a_k} h_\ell$
with some polynomials $a_k, b_{\ell k},c_{\ell k}\in P$
such that $a_k \notin \M$. 
Then $\det(\bar{b}_{\ell k})$
is a non-zero element of the ring $P/\Q$
which generates the socle of this ring (cf.~\cite{Wie}, Satz~3).
Moreover, we have the equality 
$f_{j_k} = \sum_{m=1}^n (\sum_{\ell=1}^n
b_{m\ell}\tfrac{c_{\ell k}}{a_k})\, g_m$. 
By the hypothesis, also the residue class of the element 
$\det((\sum_{\ell=1}^n b_{m\ell}
\tfrac{c_{\ell k}}{a_k}^{\mathstrut})_{m,k})
= \det(b_{\ell k})\cdot \det(\tfrac{c_{\ell k}}{a_k})$
is non-zero in $P/\Q$.
Since we know that $\det(\bar{b}_{\ell k})$ generates the socle of 
$P/\Q$,  the element $\det(\tfrac{\bar{c}_{\ell k}}{\bar{a}_k}^{\mathstrut})$ 
is invertible in this ring. Hence we have $\det(\tfrac{c_{\ell k}}
{a_k}) \notin \M P_{\M}$, and thus the matrix $(\tfrac{c_{\ell k}}{a_k})_{\ell,k}$
is invertible over~$P_\M$. Consequently, also $\{f_{j_1},\dots,f_{j_n}\}$
generates $\Q P_\M$, and the fact that $(f_{j_1},\dots,f_{j_n})$ 
is a regular sequence in $P_\M$ follows from~\cite{BH}, Theorem~2.1.2.

To prove the last claim, we let $\Q'=\langle f_{j_1},\dots,f_{j_n}\rangle$.
If the two equivalent conditions are satisfied, we have $\Q'P_\M = \Q P_\M$. 
Then $\Rad(\langle f_{j_1},\dots,f_{j_n}\rangle){=\M}$ means 
that~$\Q'$ is an $\M$-primary ideal, and we conclude that $\Q' = \Q$.
\end{proof}

Let us see some examples which illustrate
Algorithm~\ref{alg:CheckLCI} and Proposition~\ref{prop:whichgens}.

\begin{example}\label{ex:EasyCI}
Let $P=\QQ[x]$, $\M = \langle x \rangle$, 
and $\Q = \langle x(x-1),\; x(x-2)\rangle$.
Let us apply Algorithm~\ref{alg:CheckLCI} to check whether the
ring $R=P/\Q$ is a local complete intersection at~$\M$.

The residue classes modulo~$\Q$ of the $1$-minors  
of $W$ are $-1$ and $-2$. Hence we deduce that $\Q P_\M$ 
is generated by a regular sequence.
Since $\Q = \M$, it is clear that both $x(x-1)$ and $x(x-2)$ 
generate $\Q P_\M$. However, neither of these
polynomials generates~$\Q$.
Moreover, observe that $\langle x(x-1)\rangle = \Q\cap \langle x-1 \rangle$
and $\langle x(x-2)\rangle = \Q\cap \langle x-2 \rangle$.
\end{example}

\begin{example}\label{ex:LocalCaseCI}
Let $K=\QQ$, let $P=K[x,y,z]$, and let~$\X$ be the \hbox{0-di}\-men\-sional
subscheme of~$\AA^2_K$ defined by the ideal $I_\X=\langle f_1,\dots,f_4\rangle$,
where $f_1=z^2 -y$, $f_2=x^2 -2xz +y$, $f_3=yz -z -1$, and $f_4=y^2 -y -z$.

The calculation of the primary decomposition of~$I_\X$
yields that~$I_\X$ is a primary ideal and its radical is the maximal ideal
$\M=\langle x-z,\, y-z^2,\allowbreak z^3-z-1\rangle$. Here the polynomials
$g_1=x-z$, $g_2=y-z^2$, and $g_3= z^3-z-1$ form a regular sequence
which generates~$\M$. Let us use Algorithm~\ref{alg:CheckLCI} 
to check whether~$\X$ is a complete intersection scheme.

Thus we represent $f_1,\dots,f_4$ as required by Step~(2) of the
algorithm and get the matrix
$$
W = \left( \begin{array}{rccc}
\;0 \quad   & x-z \quad   & 0 \quad & 0\\
{-}1 \quad   & 1  \quad   & z \quad & z^2 +y -1\\
\;0  \quad   & 0 \quad    & 1 \quad & z
\end{array} \right)
$$

The tuple of residue classes in the ring $P/I_\X$ of the minors of order~3 of~$W$ is
$(\bar{x}-\bar{z},\; \bar{x}\bar{z}-\bar{y},\; 0,\; -\bar{x}\bar{y}+\bar{x}+1)$.
Therefore we deduce from Algorithm~\ref{alg:CheckLCI} that 
the scheme~$\X$ is locally a complete intersection.

By Proposition~\ref{prop:whichgens}, there are three triples of 
polynomials $\{f_1, f_2, f_3\}$, $\{f_1, f_2, f_4\}$, $\{f_2, f_3, f_4\}$
with the property that each triple forms
a regular sequence in $P_\M$ which generates $I_\X P_\M$, since the 
residue classes of the corresponding minors are non-zero in $P/I_\X$. 

It is interesting to observe that 
the two triples $\{f_1, f_2, f_3\}$ and $\{f_2, f_3, f_4\}$ 
generate ideals whose radical is $\M$,   hence they 
generate $I_\X$ according to the last claim 
of Proposition~\ref{prop:whichgens}.
On the other hand, the triple $\{f_1, f_2, f_4\}$ does not 
generate $I_\X$, since we have the equality
$\langle f_1, f_2, f_4\rangle = I_\X \cap \langle z,  y,  x^2\rangle$.

\end{example}

\begin{example}\label{ex:KK}
Let $K=\QQ$, let $P=K[x,y]$, and let~$\Q=\langle f_1, f_2,f_3\rangle$,
where $f_1=y^3-x^2$, $f_2=x^3 -x^2y$, and $f_3=x^2y^2$.
It turns out that the ideal~$\Q$ is a $0$-dimensional $\M$-primary ideal
for the maximal ideal $\M = \langle x,\, y\rangle$.
We represent $f_1, \ f_2, \ f_3$ as required by Step~(2) of 
Algorithm~\ref{alg:CheckLCI} and get the matrix
$$
W = \left( \begin{array}{ccc}
\;0 \quad   & x^2-xy \quad   & xy^2 \\
{-}y^2 \quad   & 0  \quad   & 0 
\end{array} \right)
$$
The tuple of residue classes in the ring $P/\Q$ of the minors of 
order~2 of~$W$ is $(-\bar{x}^2\bar{y}, \, 0, \, 0)$.
Therefore $(f_1,\,  f_2)$ is a regular sequence in $P_\M$ which generates
$\Q P_\M$. Notice that $\{f_1, \, f_2\}$ does not generate~$\Q$, 
since we have the equality $\langle f_1, \, f_2 \rangle 
= \Q \cap \langle x-1,\, y-1\rangle$.
\end{example}

\bigbreak
%
%

\section{Strict Complete Intersections via Macaulay Bases}
\label{Strict Complete Intersections via Macaulay Bases}

As in the preceding sections, we work over an arbitrary field~$K$,
we let~$\X$ be a 0-dimensional subscheme of~$\AA^n_K$,
and we denote the vanishing ideal of~$\X$ in~$P$ by $I_\X$.
Recall that a 0-dimensional local ring~$S$ is said to be
{\bf Gorenstein} if its {\bf socle}, i.e., the annihilator
of its maximal ideal~$\mathfrak{n}$, is a 1-dimensional 
$S/\mathfrak{n}$-vector space. In this setting, we consider
the following properties of~$\X$.

\begin{definition}\label{def:strictGorstrictCI}
Let $\X$ be a $0$-dimensional subscheme of~$\AA^n_K$. 
\begin{enumerate}
\item[(a)] The scheme~$\X$ is called a {\bf strict Gorenstein scheme} 
if the associated graded ring 
$\grF(R_\X)\cong P/\DF(I_\X)$ is Gorenstein.

\item[(b)] The scheme~$\X$  is called a {\bf strict complete
intersection scheme} if the degree form ideal $\DF(I_\X)$ is generated by a 
homogeneous regular sequence.
\end{enumerate}
\end{definition}

Notice that in our setting the associated graded ring
$\grF(R_\X)$ is a \hbox{0-di}\-men\-sional local ring with maximal
ideal $\langle \bar{x}_1,\dots,\bar{x}_n\rangle$.
The following definition is useful to describe the
vanishing ideal of a strict complete intersection.

\begin{definition}\label{def:strictRS}
Let $f_1, \dots, f_r \in P\setminus K$. Then $(f_1, \dots, f_r)$ is said to be 
a \textbf{strict regular sequence} if $(\DF(f_1), \dots, \DF(f_r))$ is 
a homogeneous regular sequence.
\end{definition}

Let us collect some basic properties of $0$-dimensional strict
complete intersection schemes.

\begin{proposition}\label{prop:strictCI}
Let $\X$ be a 0-dimensional subscheme of~$\AA^n_K$.
\begin{enumerate}
\item[(a)] If~$\X$ is a strict complete intersection scheme, then there
exist $f_1,\dots,f_n \in I_\X$ such that
$\DF(I_\X)=\langle \DF(f_1),\dots,\DF(f_n)\rangle$.

\item[(b)] Let $f_1, \dots, f_n \in I_\X$ be a strict regular sequence 
such that we have $\DF(I_\X)=\langle \DF(f_1),\dots,\DF(f_n)\rangle$.
Then $(f_1,\dots, f_n)$ is a regular sequence which generates~$I_\X$.
In particular, every strict complete intersection scheme~$\X$ 
is locally a complete intersection scheme.

\item[(c)] Every strict complete intersection scheme~$\X$ is a 
strict Gorenstein scheme, and every locally complete intersection scheme~$\X$ 
is a locally Gorenstein scheme.
\end{enumerate}
\end{proposition}

\begin{proof}

First we show~(a).
Since every non-zero homogeneous
element of~$\DF(I_\X)$ is a degree form, it follows that $\DF(I_\X)$
can be generated by a regular sequence of the form $(\DF(f_1),\dots,
\DF(f_n))$.

The first claim in~(b) is, for instance, shown in~\cite{KK}, Lemma 1.8 and
Prop.~1.9.  The second claim is then a consequence of the observation that~$R_\X$
is semi-local and its local rings satisfy
$P/\Q_i \;\cong\; (R_\X)_{\m_i} \;\cong\; 
P_{\M_i} /\langle f_1,\dots,f_n\rangle P_{\M_i}
$.

Finally, to prove~(c), we note that every
0-dimensional local complete intersection ring is Gorenstein
(see for instance~\cite{BH}, Prop.~3.1.20).
\end{proof}

Recall that  the Castelnuovo function of~$\X$ is said 
to be {\bf symmetric} if
we have $\Delta \HFa_\X(\ri(R_\X)-i) = \Delta \HFa_\X(i)$
for all $i\in\ZZ$.
Now we are ready to formulate the following algorithm.

\begin{algorithm}{\bf (Checking Strict Complete Intersections, 
I)}\label{alg:CheckSCI}\\
Let $\X$ be a 0-dimensional scheme in~$\AA^n_K$, and let
$R_\X = P/I_\X$ be the affine coordinate ring of~$\X$. 
The following instructions define an algorithm which
checks whether~$\X$ is a strict complete intersection scheme
and returns the corresponding Boolean value.
\begin{enumerate}
\item[(1)] Compute a Macaulay basis of~$I_\X$, i.e., compute
polynomials $\{f_1,\dots,f_r\}$  in~$I_\X$ whose degree forms generate
the degree form ideal $\DF(I_\X)$.

\item[(2)] Check if the Castelnuovo function
of~$\grF(R_\X)\cong P/\DF(I_\X)$ is symmetric.
If this is not the case, return {\tt FALSE} and stop.

\item[(3)] For $j=1,\dots,r$, write $\DF(f_j) = 
\sum_{i=1}^n a_{ij}x_i$. Form the matrix~$W$ of size 
$n\times r$ whose columns are given by 
$\sum_{i=1}^n a_{ij} e_i$ for $j=1,\dots,r$.

\item[(4)] Calculate the tuple of the residue classes in~$P/\DF(I_\X)$
of the minors of order~$n$ of~$W$. If one of these residue 
classes is non-zero,
return {\tt TRUE}. Otherwise, return {\tt FALSE}.
\end{enumerate}
\end{algorithm}

\begin{proof}
As every strict complete intersection is a strict Gorenstein
scheme, Step~(2) produces the correct answer if the
affine Hilbert function of~$\X$ is not symmetric.
By the preceding proposition, Steps~(3) and~(4) compute a tuple
of elements of~$P/\DF(I_\X)$ whose entries generate
the 0-th Fitting ideal of the maximal ideal $\langle
\bar{x}_1,\dots,\bar{x}_n\rangle$ of that ring.
Thus the correctness of the answer in Step~(4) follows from
Proposition~\ref{prop:gensyz}.
\end{proof}

In Proposition~\ref{prop:whichgens} we described a condition under which
we can obtain  a regular sequence which generates $\Q P_\M$.
Here we can do better.

\begin{corollary}\label{cor:whichhomoggens}
In the setting of Algorithm~\ref{alg:CheckSCI}, the following conditions 
are equivalent.
\begin{enumerate}
\item[(a)] The residue class of the
minor involving columns $(j_1,\dots, j_n)$ of the matrix~$W$
is non-zero.

\item[(b)] The polynomials $(f_{j_1},\dots,f_{j_n})$
form a strict regular sequence which generates the ideal~$I_\X$.
\end{enumerate}
\end{corollary}

\begin{proof}
The proof is analogous to the proof of Proposition~\ref{prop:whichgens}. 
The main difference is that in the proof of (a)$\Rightarrow$(b)
of that proposition we show that $\det(c^{\mathstrut}_{\ell k})$
is not contained in $\M P_\M$.
On the other hand, in the current setting the determinant 
$\det(c^{\mathstrut}_{\ell k})$ is a homogeneous polynomial, 
and~$\M$ is the ideal generated by the indeterminates.
Consequently, $\det(c^{\mathstrut}_{\ell k})$ is a non-zero constant,
and hence invertible in~$P$.
\end{proof}

The next examples show Algorithm~\ref{alg:CheckSCI} 
and Corollary~\ref{cor:whichhomoggens} at work.

\begin{example}\label{ex:8pointscubic}
Let $K=\QQ$, let $P=K[x,y,z]$, and let~$\X$ be the reduced
subscheme of~$\AA^3_K$ consisting of the
eight points $p_1 = (0,0,0)$, $p_2 =(1,1,1)$, $p_3 = (-1,1,-1)$,
$p_4 = (2,4,8)$, $p_5 = (-2,4,-8)$, $p_6 = (3,9,27)$, $p_7 = (-3,9,-27)$,
and $p_8 = (4,16,64)$ on the twisted cubic curve
$T=\{(t, t^2, t^3)\in\AA^3_K \mid t\in K\}$.
The reduced $\tt DegRevLex$-Gr\"obner basis
of the vanishing ideal~$I_\X$ is
$$
\begin{array}{l}
\{y^2 -xz,\  xy -z,\  x^2 -y,\;
\smallskip\\
yz^2 -\tfrac{2}{15} z^3 +49xz +\tfrac{98}{5}yz -14z^2
+\tfrac{336}{5}x -36y -\tfrac{260}{3} z,
\smallskip\\
xz^2 -\tfrac{1}{30} z^3 -\tfrac{91}{10} yz
-\tfrac{96}{5} x +\tfrac{82}{3} z,
\smallskip\\
z^4 -\tfrac{418}{5} z^3 +6699xz +\tfrac{61446}{5} yz
-\! 1408z^2 +\tfrac{210672}{5} x -\! 5292y -\! 54340z\}
\end{array}
$$
and hence
$\DF(I_\X) = \langle y^2 -xz,\  xy,\  x^2, \ yz^2 -\tfrac{2}{15} z^3, \
xz^2 -\tfrac{1}{30} z^3, \  z^4\rangle$.

In Step~(2) we calculate the Castelnuovo function $(1,3,3,1)$
and notice that it is symmetric.

Next we write the generators of~$\DF(I_\X)$ in the form
$y^2-xz = (-z)x + (y)y + (0)z$, etc., and get the matrix
$$
W= \left(\begin{array}{rcccccc}
-z \quad & y \quad  &  x\quad   & 0 & 0 &0 \\
y \quad &  0 \quad &0\quad  & 0 & 0 & 0\\
0 \quad &  0\quad  & 0\quad  & yz-\tfrac{2}{15}z^2\quad 
& xz - \tfrac{1}{30} z^2\quad & z^3
\end{array} \right)
$$
Since the residue classes of all maximal minors of~$W$ are zero in
$P/\DF(I_\X)$, we conclude that $\X$ is not a strict complete 
intersection scheme. Notice, however, that we can use~\cite{KLR},
Algorithm~4.6 and Theorem~6.8 to check that~$\X$
is a strict Gorenstein scheme.
\end{example}

A positive answer is obtained in the following case.

\begin{example}\label{ex:CI}
Let $K=\QQ$, let $P=K[x,y]$, 
let $f_1 = x^3 -x -2y^5 +4y^4 -2y^3 +4y^2 -1$, $f_2 = 
xy -y^5 +2y^4 -y^3 +2y^2$, and $f_3 = y^7 -4y^6 +5y^5 -4y^4 +4y^3 -y$,
and let $\X$ be the 0-dimensional
subscheme of~$\AA^2_K$ defined by
$I_\X = \langle f_1,\ f_2, \ f_3\rangle$.
Then the reduced ${\tt DegRevLex}$-Gr\"obner basis of~$I_\X$
turns out to be
$$
G= (y^5 -2y^4 +y^3 -xy -2y^2,\;  x^3 -2xy -x -1,\;  xy^3 -2xy^2 -y,\;
x^2y -y^3 -y)
$$
and hence we have
$\DF(I_\X) = \langle y^5,\  x^3,\  xy^3, \ x^2y-y^3\rangle$.

In Step~(2) we get the Castelnuovo function
$(1,\,2,\, 3,\, 2,\, 1)$ which is symmetric.
Next we write $y^5 = (0)x+(y^4)y$, etc., and get the matrix
$$
W = \left(\begin{array}{lllc}
0 \ & x^2 \ & y^3\ &  0\\
y^4 \ & 0 \ & 0 \ & x^2-y^2
\end{array} \right)
$$
Here the tuple of residue classes in~$P/\DF(I_\X)$
of the minors of order~2 of~$W$ is given by
$(0,\; 0,\; 0,\; 0,\; -\bar{y}^4,\; 0)$,
and therefore~$\X$ is a strict complete intersection scheme.

Notice that the non-zero minor of order~2 is obtained by selecting
the second and the fourth column of~$W$.
The corresponding polynomials satisfy 
$\DF(g_2)=x^3$ and $\DF(g_4)=x^2y-y^3$. They
form a regular sequence which generates $\DF(I_\X)$. 
In particular, note that we have $y^5 = (xy)x^3 +(-x^2 -y^2)(x^2y -y^3)$
and $xy^3 = (y)x^3 + (-x)(x^2y -y^3)$.

Hence 
the polynomials $g_2=x^3 -2xy -x -1$ and $g_4=x^2y -y^3 -y$
form a strict regular sequence. It turns out that 
$$
\begin{array}{cl}
f_1 &\!\!= (-2xy +1)\, g_2 + (2x^2 +2y^2 -4y)\, g_4\\
f_2 &\!\!= (-xy)\,g_2 + (x^2 +y^2 -2y)\,g_4\\
f_3 &\!\!= (xy^3 -2xy^2 +y)\,g_2 + (-x^2y^2 -y^4 +2x^2y +4y^3 -4y^2 -x)\,g_4
\end{array}
$$
\end{example}

\bigbreak
%
%

\section{Strict Complete Intersections via Border Bases}
\label{Strict Complete Intersections via Border Bases}

In this section we continue to work in the setting
defined at the beginning of Section~\ref{Zero-Dimensional Affine Schemes}.
Our goal is to construct another algorithm for checking 
the strict complete intersection property which is based
on the computation of a degree filtered border basis of~$I_\X$.
This algorithm has the advantage of allowing us to check
which ideals in certain families of ideals are strict
complete intersections (see Example~\ref{ex:SCI22} and~\cite{KLR3}).
First of all, we define the following useful kind of vector 
space bases of~$R_\X$.

\begin{definition}\label{def:degfilt}
A tuple $B=(b_1,\dots, b_\mu) \in R_\X^\mu$ is said to be
a {\bf degree filtered $K$-basis} of~$R_\X$ if
the set $F_iB = B\cap F_i R_\X$ is a $K$-basis of~$F_i R_\X$
for every $i\in\ZZ$ and if $\ord_\F(b_1) \le \cdots \le \ord_\F(b_\mu)$.
\end{definition}

In addition to these properties, we may (and shall) assume that
$b_1=1$ in each degree filtered $K$-basis of~$R_\X$.
For a discussion of this notion, we refer to~\cite{KLR}, 
Remark~2.10 and Example~2.11.
Degree filtered bases of~$R_\X$ can  be characterized as follows.

\begin{proposition}\label{prop:CharDFB}
Let~$\X$ be a 0-dimensional subscheme of~$\AA^n_K$ as above, and
let $B=(b_1,\dots,b_\mu)\in R_\X^\mu$ with $\ord_\F(b_1) \le \cdots
\le \ord_\F(b_\mu)$.
Then the following conditions are equivalent.
\begin{enumerate}
\item[(a)] The tuple~$B$ is a degree filtered $K$-basis of~$R_\X$.

\item[(b)] The leading form tuple $\LF(b_1),\dots,\LF(b_\mu)$
is a $K$-basis of the graded ring $\grF(R_\X)\cong P/\DF(I_\X)$.

\item[(c)] There exist polynomials $f_1,\dots,f_\mu\in P$
such that $b_i=f_i+I_\X$ for $i=1,\dots,\mu$ and
$B\hom = (f_1\hom + I_\X\hom,\dots, f_\mu\hom + I_\X\hom)$
is a $K[x_0]$-module basis of~$R_\X\hom$.
\end{enumerate}
\end{proposition}

\begin{proof}
To show that~(a) implies~(b), we note that 
$\dim_K(\grF(R_\X))=\dim_K(R_\X)$
implies that it suffices to prove linear independence.
Assume that there exists a tuple
$(a_1,\dots,a_\mu)\in K^\mu \setminus \{0\}$ such that we have
$a_1 \LF(b_1) + \cdots + a_\mu \LF(b_\mu){=0}$. Since the elements
$\LF(b_i)$ are homogeneous in the graded ring $\grF(R_\X)$, we may assume
that there is a degree~$d$ such that $a_i=0$ for $\ord_\F(b_i)\ne d$.
Let $k, k+1,\dots,\ell\in\{1,\dots,\mu\}$ be the indices for which
$\ord_\F(b_i)=d$. Then $a_k \LF(b_k) + \cdots + a_\ell \LF(b_\ell)=0$
implies $a_k b_k + \cdots +a_\ell b_\ell \in F_{d-1}R_\X = \langle b_1,
\dots,b_{k-1}\rangle_K$, in contradiction
to the linear independence of~$F_d B = (b_1,\dots,b_\ell)$.

To prove that~(b) implies~(c), we choose $f_i\in P$ such that
$b_i = f_i+ I_\X$ and $\deg(f_i)= \ord_\F(b_i)$ for $i=1,\dots,\mu$.
(If $f_i$ does not satisfy the second condition, replace it by
$\NF_{\sigma,I_\X}(f_i)$ for some degree compatible term ordering~$\sigma$.)
Then we have $\LF(b_i)=\DF(f_i) + \DF(I_\X)$ for $i=1,\dots,\mu$.
Now the claim is a consequence of Remark~\ref{ReesAlg}.c
and $\DF(f_i) = f_i\hom(0,x_1,\dots,x_n)$ for $i=1,\dots,\mu$.

Finally we show that~(c) implies~(a). In view of the isomorphisms in
Remark~\ref{ReesAlg}.c, the elements $b_i= f_i+I_\X =
f_i\hom(1,x_1,\dots,x_n) + I_\X$ with $i\in\{1,\dots,\mu\}$
form a $K$-basis of~$R_\X$, and the elements
$$
\LF(b_i)= \DF(f_i)+ \DF(I_\X) = f_i\hom(0,x_1,\dots,x_n) + \DF(I_\X)
$$
with $i\in\{1,\dots,\mu\}$ form a $K$-basis of $\grF(R_\X)$.
From the latter property and the formula $\HFa_\X(i) =
\sum_{j=0}^i \HF_{\grF(R_\X)}(j)$
we deduce that the number of elements
in $F_i B$ equals $\HFa_\X(i)$ for all $i\ge 0$. Hence
$F_i B$ is a $K$-basis of $F_i R_\X$ for all $i\ge 0$, and
the proof is complete.
\end{proof}

An even more useful type of degree-filtered
$K$-bases is described in the following proposition.
We recall that $\mathbb{T}^n = 
\{ x_1^{\alpha_1}\cdots x_n^{\alpha_n} \mid \alpha_i \ge 0\}$
is the monoid of terms in~$P$.

\begin{proposition}\label{prop:CharDFBB}
Let $\OO=\{t_1,\dots,t_\mu\}$ be an order ideal of terms
in~$\mathbb{T}^n$ such that $\deg(t_1)\le \cdots
\le\deg(t_\mu)$. Let $\partial\OO=\{b_1,\dots, b_\nu\}$
be the border of~$\OO$, and let $G=\{g_1,\dots,g_\nu\}$
be an $\OO$-border prebasis of the ideal~$I_\X$, where we have
$g_j= b_j - \sum_{i=1}^\mu \gamma_{ij} t_i$
with $\gamma_{1j},\dots,\gamma_{\mu j}\in K$ for $j=1,\dots,\nu$.
Then the following conditions are equivalent.
\begin{enumerate}
\item[(a)] The terms in~$\OO$ form a degree filtered $K$-basis
of~$R_\X$.

\item[(b)] The terms in~$\OO$ form a $K$-basis of~$\grF(R_\X)
\cong P/\DF(I_\X)$.

\item[(c)] The ideal~$I_\X$ has an $\OO$-border basis and
for all $i\ge 0$ we have
$\HFa_{\X}(i)= \#\{j\in\{1,\dots,\mu\} \mid \deg(t_j)=i\}$.

\item[(d)] The ideal~$I_\X$ has an $\OO$-border basis and we have
$b_j\in \DF(g_j)$ for $j=1,\dots,\nu$.

\item[(e)] The polynomials in $\DF(G)=\{\DF(g_1),\dots,
\DF(g_\nu)\}$ form an $\OO$-border basis of~$\DF(I_\X)$,
hence they generate $\DF(I_\X)$.
\end{enumerate}
If these conditions are satisfied, we say that~$G$ is a
{\bf degree filtered $\OO$-border basis} of~$I_\X$.
\end{proposition}

\begin{proof}
The equivalence of~(a) and~(b) was shown in Proposition~\ref{prop:CharDFB}.
The equivalence of~(a) and~(c) follows immediately from the definitions.
The equivalence of~(b) and~(d) follows from~\cite{KR4}, Thm.~2.4.
The implication (b)$\Rightarrow$(e) was also shown 
in~\cite{KR4}, Thm.~2.4, and (e)$\Rightarrow$(b) holds by definition.
\end{proof}

Using a degree filtered border basis as described in the above proposition, 
we can formulate
another version of Algorithm~\ref{alg:CheckSCI} for checking the strict
complete intersection property.

\begin{algorithm}{\bf (Checking Strict Complete Intersections,
II)}\label{alg:CheckSCI2}\\
Let $\X$ be a 0-dimensional scheme in~$\AA^n_K$, let
$R_\X = P/I_\X$ be the affine coordinate ring of~$\X$, and let
$\mu=\dim_K(R_\X)$. The following instructions define an algorithm 
which checks whether~$\X$ is a strict complete intersection and
returns the corresponding Boolean value.
\begin{enumerate}
\item[(1)] Compute an order ideal $\OO=(t_1,\dots,t_\mu)$
in~$\mathbb{T}^n$ such that~$I_\X$ has a degree filtered
$\OO$-border basis.  Let $\rho=\deg(t_\mu)$.

\item[(2)] For $i=0,\dots,\rho$, find
$h_i = \# \{j\in\{1,\dots,\mu\} \mid \deg(t_j)=i\}$.

\item[(3)] For $i=0,\dots, \lfloor \rho/2 \rfloor$, check
whether $h_{\rho-i}=h_i$. If this is not the case, return
{\tt FALSE} and stop.

\item[(4)] Let $G=\{g_1,\dots,g_\nu\}$ be the $\OO$-border basis of~$I_\X$.
For $j=1,\dots,\nu$, write $\DF(g_j)= \sum_{i=1}^n h_{ij} x_i$
with homogeneous polynomials $h_{ij}\in P$.

\item[(5)] Form the matrix $W$ of size $n\times \nu$
whose columns are given by $\sum_{i=1}^n h_{ij} e_i$ 
for $j=1,\dots,\nu$.

\item[(6)] Calculate the tuple of residue classes in
$P/\DF(I_\X)$ of the minors of order~$n$ of~$W$.

\item[(7)] If the result is different from $(0,\dots,0)$,
return {\tt TRUE}. Otherwise return {\tt FALSE}.
\end{enumerate}
\end{algorithm}

\begin{proof}
By Step~(1), the residue classes of the elements of~$\OO$
form a $K$-basis of the ring $P/\DF(I_\X)$.
Since every strict complete intersection is a strict Gorenstein
ring, we can stop in Step~(3) if the Castelnuovo function
$(h_0,\dots,h_\rho)$ of this ring is not symmetric.

Now we use Algorithm~\ref{alg:CheckSCI} for the ring $P/\DF(I_\X)$.
Notice that the degree forms $\DF(g_1),\dots,\DF(g_\nu)$ 
generate $\DF(I_\X)$ by Proposition~\ref{prop:CharDFBB}.e.
In Step~(4) we write them as linear combinations
of~$x_1,\dots,x_n$. Hence the correctness follows from
Algorithm~\ref{alg:CheckSCI}.
\end{proof}

Next example shows how the methods described in this paper 
allow us to detect strict complete intersections within certain 
families of $0$-dimensional schemes.
For a more detailed discussion of this topic see~\cite{KLR3}.

\begin{example}\label{ex:SCI22}
Let $K$ be a field, let $P=K[x,y]$, and let~$\OO$ be the order ideal 
$\OO = \{1, y, x, xy\}$.
According to~\cite{KR4}, Example 3.8, the 
ideal $I = \langle f_1, f_2, f_3, f_4\rangle$
where
\begin{eqnarray*}
f_1 &=& y^2 - (-c_{23}c_{41}c_{42} + c_{21}c_{42}c_{43} - c_{21}c_{44}+  c_{23})
\\ && - c_{21}x - (-c_{21}c_{42} - c_{41}c_{44} + c_{43})y
- c_{41}xy,\\
f_2 &=& x^2 - (-c_{34}c_{41}c_{42} + c_{32}c_{41}c_{44} - c_{32}c_{43}+ c_{34})
\\ && - (-c_{32}c_{41} -  c_{42}c_{43} +  c_{44})x - c_{32}y
- c_{42}xy,\\
f_3 &=& xy^2 - (c_{23}c_{32}c_{41} -  c_{21}c_{32}c_{43} + c_{21}c_{34})
\\ && - c_{23}x - (c_{21}c_{32} + c_{34}c_{41})y - c_{43}xy,\\
f_4 &=& x^2y -(c_{21}c_{34}c_{42} - c_{21}c_{32}c_{44} + c_{23}c_{32})
\\ && - (c_{21}c_{32}  + c_{23}c_{42})x - c_{34}y - c_{44}xy,\\
\end{eqnarray*}
is the defining ideal of the universal family of all subschemes of
length four of the affine plane which have the property that their 
coordinate ring admits $\overline{\mathcal{O}}$
as a vector space basis. The parameters $c_{21}$, $c_{23}$, $c_{32}$, 
$c_{34}$, $c_{41}$, $c_{42}$, $c_{43}$, and~$c_{44}$ are free. 
In other words, the family is parametrized by an $8$-dimensional affine space.
Since the degree form ideal $\DF(I)$ is generated by the degree forms 
of $f_1, f_2, f_3, f_4$, 
we have $\DF(I) =  \langle y^2 - c_{41}xy, \ x^2 - c_{42}xy, \ xy^2,\ x^2y\rangle$.
To compute the locus of strict complete intersections 
in~$\mathbb{A}^8_K$, we write the generators of~$\DF(I)$ in the form
$$
(- c_{41}y)x +(y)y,\quad  (x)x +(- c_{42}x)y,\quad (y^2)x +(0)y, 
\quad (xy)x + (0)y
$$
We get the matrix 
$$
W = \left(\begin{array}{cccc}
-c_{41}y \quad &             x\quad &     y^2\quad &  xy\\
y\  \quad      &   c_{42}x \quad &      0 \quad &    0
\end{array} \right)
$$
Then the only non-zero maximal minor of $W$ modulo $\DF(I)$ is given by
$(1-c_{41}c_{42})xy$. In conclusion, outside the hypersurface 
in~$\AA^8_K$ defined by $1-c_{41}c_{42} = 0$,
all ideals define a strict complete intersection scheme.
\end{example}

Finally, we note that one can also use the K\"ahler different
of~$R_\X$ to check whether this ring is a strict complete intersection
or locally a complete intersection. However, this approach introduces
constraints on the characteristic of the base field.
Let us formulate the characterizations underlying the
algorithms using K\"ahler differents and leave the details
to the interested reader.

\begin{remark}\label{rem:Jac}
Let $\X$ be a 0-dimensional subscheme of~$\AA^n_K$, and let 
$R_\X=P/I_\X$ be the affine coordinate ring of~$\X$.
\begin{enumerate}
\item[(a)] The module of K\"ahler differentials
$\Omega^1_{R_\X/K}$ is given by the presentation
$$
\Omega^1_{R_\X/K} \cong {\textstyle\bigoplus\nolimits_{i=1}^n}
P dx_i / \left( I_\X \cdot {\textstyle\bigoplus\nolimits_{i=1}^n}
P dx_i + \langle {\textstyle\sum\nolimits_{i=1}^n}
\partial f/\partial x_i \; dx_i \mid f\in I_\X \rangle \right)
$$

\item[(b)] The K\"ahler different $\theta_{R_\X}$ of the $K$-algebra $R_\X$
is the ideal in~$R_\X$
generated by residue classes of the maximal minors
of the {\bf Jacobian matrix}
$\Jac(f_1,\dots,f_r) = (\partial f_i / \partial x_j)_{i,j}$,
where $\{f_1,\dots,f_r\}$ is a system of generators of~$I_\X$.
For further details about K\"ahler differents, see~\cite{Ku2}, \S~10.

\item[(c)] The K\"ahler different $\theta_{\grF(R_\X)}$ of the associated 
graded ring $\grF(R_\X) \cong P/\DF(I_\X)$ is defined similarly.

\item[(d)] Suppose that $\charac(K)$ does not divide~$\mu$.
Then $\X$ is a strict complete intersection if and only if
$\theta_{\grF(R_\X)}$ is non-zero.

\item[(e)] Again, suppose that $\charac(K)$ does not divide~$\mu$.
Then~$\X$ is locally a complete intersection if and only if
the image of $\theta_{R_\X}$ in every local ring of~$R_\X$
is non-zero. (For instance, if we know the principal idempotents
$f_1,\dots,f_s$ of~$R_\X$, it suffices to check that
$f_i\cdot \theta_{R_\X}\ne \langle 0\rangle$ for $i=1,\dots,s$.)

\end{enumerate}
\end{remark}

The following easy example shows that the assumption on the
characteristic of~$K$ is necessary for the approach
via K\"ahler differents to work, while the approach based
on Wiebe's result (see Proposition~\ref{prop:Wiebe})
works in general.

\begin{example}\label{ex:charCI}
Let $p$ be a prime number,  let $K=\mathbb F_p$, let $P=K[x]$, and let~$\X$
be the 0-dimensional subscheme of~$\AA^1_K$ defined by $I_\X = \langle x^p \rangle$.

When we use Algorithm~\ref{alg:CheckSCI} to check whether~$\X$ is a strict
complete intersection, we find the matrix $W=(x^{p-1})$ whose
determinant yields the relation $\bar{x}^{p-1}\in \grF(R_\X)\setminus \{0\}$.
Hence we conclude that~$\X$ is a strict complete intersection.

However, the Jacobian matrix is $\Jac(x^p)=(0)$ and therefore we have
$\theta_{\grF(R_\X)}=\langle 0\rangle$. Thus the K\"ahler different
does not yield the correct answer.
\end{example}

\medskip

\noindent{\bf Acknowledgements.}
The first and second authors were partially supported by
the Vietnam National Foundation for Science and 
Technology Development (NAFOSTED) grant number 101.04-2019.07.
The third author thanks the University of Passau 
for its hospitality and support during part of the preparation
of this paper.

\medskip

\bibliographystyle{plain}

\begin{thebibliography}{99}


\bibitem{ABPR} Abbott, J., Bigatti, A.M.,
Palezzato, E., Robbiano, L.: Computing and using minimal
polynomials, J.\ Symb.\ Comput., 2019 (to appear)

\bibitem{CoCoA} Abbott, J., Bigatti, A.M., 
Robbiano, L.: \cocoa : a system for doing Computations in 
Commutative Algebra, {\tt http://cocoa.dima.unige.it}.

\bibitem{As59} Assmus, E.F.: On the homology of local
rings, Illinois J.\ Math.\ {\bf 3} (1959), 187--199.

\bibitem{AH94} Avramov, L.L.,
Herzog, J.: Jacobian criteria for complete intersections. {T}he
graded case, Invent.\ Math.\ {\bf 117} (1994), 75--88.

\bibitem{BGS} Bermejo, I., Garcia-Marco, I.,
Salazar-Gonzalez, J.J.: An algorithm for checking whether
the toric ideal of an affine monomial curve is a complete
intersection, J.\ Symb.\ Comput.\ {\bf 42} (2007), 971--991.

\bibitem{BH} Bruns, W., Herzog, J.: 
Cohen-Macaulay Rings, Cambridge Univ.\ Press, Cambridge, 1993.

\bibitem{GL} Gulliksen, T.H., Levin, G.:
Homology of local rings, Queen's {P}apers in {P}ure and {A}ppl.\ 
{M}ath.\ {\bf 20}, Queen's University, Kingston, 1969.

\bibitem{KK} Kreuzer, M., Kunz, E.: 
Traces in strict Frobenius algebras and strict complete intersections, 
J.\ Reine Angew.\ Math.\ {\bf 381} (1987), 181-204.

\bibitem{KL} Kreuzer, M., Long, L.N:
Characterizations of zero-dimensional complete intersections,
Beitr. Algebra Geom.\ {\bf 58} (2017), 93--129.

\bibitem{KLR} Kreuzer, M., Long, L.N.,
Robbiano, L.: On the Cayley-Bacharach property, Comm.\ Algebra {\bf 47}
(2019), 328--354.

\bibitem{KLR3} Kreuzer, M., Long L.N., Robbiano L.:
Subschemes of the border basis scheme, in preparation.

\bibitem{KR1} Kreuzer, M., Robbiano, L.:
{\it Computational Commutative  Algebra 1}, 
Sprin\-ger-Verlag, Heidelberg, 2000.

\bibitem{KR2} Kreuzer, M., Robbiano, L.:
{\it Computational Commutative Algebra 2}, Sprin\-ger-Verlag, Heidelberg, 2005.

\bibitem{KR3} Kreuzer, M., Robbiano, L.:
{\it Computational  Linear and Commutative Algebra}, Sprin\-ger-Verlag, 
Heidelberg, 2016.

\bibitem{KR4} Kreuzer, M., Robbiano, L.:
Deformations of border bases, Collect.\ Math.\ {\bf 59} (2008), 275-297.

\bibitem{Ku2} Kunz, E.: {\it K\"{a}hler Differentials}, 
Adv. Lectures Math., Vieweg Verlag, Braunschweig, 1986.

\bibitem{Lo} Long, L.N.:
Various differents for 0-dimensional schemes and applications,
dissertation, University of Passau, Passau, 2015.

\bibitem{T57} Tate, J.: Homology of {N}oetherian rings
and local rings, Illinois J.\ Math.\ {\bf 1} (1957), 14--27.

\bibitem{V67} Vasconcelos, W.V.: Ideals generated
by {R}-sequences, J.\ Algebra {\bf 6} (1967), 309--316.

\bibitem{Wie} Wiebe, H.: \"Uber homologische Invarianten 
lokaler Ringe, Math. Ann. {\bf 179} (1969), 257-274.


\end{thebibliography}

\end{document}